\documentclass[11pt]{amsart}

\usepackage{amsmath, amssymb, amscd, cancel, graphicx, soul}
\usepackage{mathdots}

\usepackage[all]{xy}
\usepackage{multirow}
\usepackage{longtable}
\usepackage{array}

\newtheorem{thm}{Theorem}[section]

\newtheorem{cor}[thm]{Corollary}
\newtheorem{lem}[thm]{Lemma}
\newtheorem{prop}[thm]{Proposition}

\newtheorem{defin}[thm]{Definition}

\def\c{{\mathfrak c}}

\def\s{{\mathfrak s}}
\def\t{{\mathfrak t}}

\def\L{{\mathbb L}}

\def\Q{{\mathbb Q}}

\def\T{{\mathcal T}}
\def\Z{{\mathbb Z}}

\def\Char{{\mathrm{Char}}}

\def\spc{{\mathrm{spin^c}}}
\def\Spc{{\mathrm{Spin^c}}}

\def\del{{\partial}}

\def\mod{{\textup{mod} \;}}
\def\rk{{\mathrm{rk}}}

\begin{document}

\title[L-space surgeries, genus bounds, and the cabling conjecture]%
{L-space surgeries, genus bounds, \\ and the cabling conjecture}

\begin{abstract}

We establish a tight inequality relating the knot genus $g(K)$ and the surgery slope $p$ under the assumption that  $p$-framed Dehn surgery along $K$ is an L-space that bounds a sharp 4-manifold.  This inequality applies in particular when the surgered manifold is a lens space or a connected sum thereof.  Combined with work of Gordon-Luecke, Hoffman, and Matignon-Sayari, it follows that if surgery along a knot produces a connected sum of lens spaces, then the knot is either a torus knot or a cable thereof, confirming the cabling conjecture in this case.

%We prove that if positive integer $p$-surgery along a knot $K \subset S^3$ produces an L-space and it bounds a sharp 4-manifold, then the knot genus obeys the bound  \[ 2g(K) -1 \leq p - \sqrt{3p+1}.\] Moreover, there exists an infinite family of pairs $(K_n,p_n)$ attaining this bound, where $K_n$ denotes an $n$-fold iterated cable of the unknot and $p_n \to \infty$.  In particular, the stated bound applies when the knot surgery produces a lens space or a connected sum thereof.  Combined with work of Gordon-Luecke, Hoffman, and Matignon-Sayari, it follows that if surgery along a knot produces a connected sum of lens spaces, then the knot is either a torus knot or a cable thereof, confirming the cabling conjecture in this case.

\end{abstract}

\author[Joshua Evan Greene]{Joshua Evan Greene}

\address{Department of Mathematics \\ Columbia University \\ New York, NY 10027}

\thanks{Partially supported by an NSF Post-doctoral Fellowship.}

\email{josh@math.columbia.edu}

\maketitle

\section{Introduction}\label{s: Lspace intro}

\subsection{Lens space surgeries.} Denote by $K$ a knot in $S^3$, $p$ a positive integer, and $q$ a non-zero integer. For a knot $K$ and slope $p/q$, let $K_{p/q}$ denote the result of $p/q$ Dehn surgery along $K$.   By definition, the lens space $L(p,q)$ is the {\em oriented} manifold $-U_{p/q}$, where $U$ denotes the unknot and $p/q \ne 1/n$.

When can surgery along a non-trivial knot $K$ produce a lens space?  This question remains unanswered forty years since Moser first raised it \cite{moser:torus}, although work by several researchers has led to significant progress on it.  For example, the cyclic surgery theorem of Culler-Gordon-Luecke-Shalen asserts that either $K$ is a torus knot or the surgery slope is an integer \cite{cgls:cyclic}, and a conjecturally complete construction due to Berge accounts for all the known examples \cite{berge:lens}.  Furthermore, we determine the complete list of lens spaces obtained by integer surgery along a knot in \cite{greene:lens}.

On the basis of Berge's construction, Goda-Teragaito conjectured an inequality relating the  surgery slope that produces a lens space and the knot genus $g(K)$ \cite{gt:lenssurgery}.  Reflect $K$ if necessary in order to assume that the slope is positive; then their conjecture asserts that for a {\em hyperbolic} knot $K$, \begin{equation}\label{e: goda-teragaito} \frac{p-1}{2} \leq 2g(K) -1 \leq p - 9. \end{equation}  The case of a non-hyperbolic knot is well-understood.  Note that $2g(K)-1$ equals {\em minus} the maximum Euler characteristic of a Seifert surface for $K$.

Both bounds in \eqref{e: goda-teragaito} are now close to settled.  Rasmussen established the inequality \[ \frac{p-5}{2} \leq 2g(K) -1 \] for {\em any} knot $K$ for which $K_p$ is a lens space, noting that it is attained for $p= 4k+3$ and $K$ the $(2,2k+1)$-torus knot \cite[Theorem 1]{r:lenssurgery}.  Kronheimer-Mrowka-Ozsv\'ath-Szab\'o established the bound \begin{equation}\label{e: kmos} 2g(K)-1 \leq p \end{equation} by an application of monopole Floer homology \cite[Corollary 8.5]{kmos}.  Their argument utilizes the fact that the Floer homology of a lens space is as simple as possible: $\rk \; \widehat{HF}(Y) = |H_1(Y;\Z)|$.  A space with this property is called an {\em L-space}, and a knot with a positive L-space surgery is called an {\em L-space knot}.  The proof of \eqref{e: kmos} extends to show that the set of positive rational slopes for which surgery along $K$ results in an L-space is either empty or consists of all rational values $\geq 2g(K) -1$.  This fact holds in the setting of Heegaard Floer homology as well \cite{os:absgr}, the framework in place for the remainder of this paper.

As remarked in \cite[pp. 537-8]{kmos}, the bound \eqref{e: kmos} can often be improved for the case of a lens space surgery.  Indeed, a closer examination of the Berge knots suggests the bound \begin{equation}\label{e: berge bound} 2g(K)-1 \leq p - 2 \sqrt{(4p+1)/5} \end{equation} whenever $K_p$ is a lens space, with the exception of $K$ the right-hand trefoil and $p=5$ (cf. \cite{saito:lens}).  This bound is attained by an infinite sequence of type VIII Berge knots $K$ and slopes $p \to \infty$.  Indeed, work of Rasmussen, including extensive computer calculations, implies that the bound \eqref{e: berge bound} holds for all $p \leq 100,000$ \cite{r:Lspace}.  

The current work addresses an improvement on the bound \eqref{e: kmos} in the direction of \eqref{e: berge bound}.  We begin with the method introduced and carried out in \cite{kmos,os:absgr}, which uses a version of Theorem \ref{t: torsion} below.  That theorem uses the {\em correction terms} of a lens space $L(p,q)$ to place a restriction on the genus of a knot $K$ with $K_p = L(p,q)$.  However, the formulae for these correction terms often prove unwieldy towards the end of extracting explicit bounds on the knot genus.  The key advance presented here stems from the observation that a lens space bounds a {\em sharp} four-manifold (Definition \ref{d: sharp}), whose existence enables us to distill the desired information.  In this more general set-up, we obtain the following result.

\begin{thm}\label{t: main}

Let $K$ denote an L-space knot and suppose that $K_p$ bounds a smooth, negative definite 4-manifold $X$ with $H_1(X;\Z)$ torsion-free.  Then the knot genus is bounded above by \begin{equation}\label{e: genusbound 1} 2g(K)-1 \leq p - \sqrt{p} - 1. \end{equation}  If $X$ can be chosen sharp, then we obtain the improved bound \begin{equation}\label{e: genusbound 2} 2g(K)-1 \leq p - \sqrt{3p+1}.\end{equation}  Furthermore, there exists an infinite family of pairs $(K_n, p_n)$ that attain equality in \eqref{e: genusbound 2}, where $K_n$ denotes an $n$-fold iterated cable of the unknot, and $p_n \to \infty$.

\end{thm}

\noindent We do not know as much concerning the tightness of inequality \eqref{e: genusbound 1}.  It does, however, lead to an improvement over \cite[Proposition 1.3]{os:lattices} for $p \geq 9$, which under the same assumptions establishes that $2g(K)-1 \leq p-4$ for the specific case of a torus knot $K$.

For the case of a lens space surgery, we establish the bound \eqref{e: berge bound} in \cite{greene:lens}.

\begin{thm}\label{t: berge bound}

Suppose that $K \subset S^3$, $p$ is a positive integer, and $K_p$ is a lens space.  Then \[2g(K)-1 \leq p - 2 \sqrt{(4p+1)/5},\] unless $K$ is the right-hand trefoil and $p = 5$.  Moreover, this bound is attained by an infinite family of distinct type VIII Berge knots $K$ and slopes $p \to \infty$. \qed

\end{thm}

\noindent  For comparison between Theorems \ref{t: main} and \ref{t: berge bound}, note that $2\sqrt{4/5} \approx 1.79$ and $\sqrt{3} \approx 1.73$.  We touch on Theorem \ref{t: berge bound} again in Section \ref{ss: lens conclusion}.

Thus, both \eqref{e: berge bound} and \eqref{e: genusbound 2} dramatically improve on Goda-Teragaito's second conjectured bound \eqref{e: goda-teragaito} for $p \gg 0$. In Section \ref{ss: goda-teragaito} we indicate how that bound follows for all except two values $p \in \{14,19\}$.  In addition, Baker-Grigsby-Hedden \cite{bgh:lens} and Rasmussen \cite{r:Lspace} have proposed programs to prove the completeness of Berge's construction using Floer homology.  One step involved in both approaches is to argue the non-existence of a non-trivial knot $K$ for which $K_{2g(K)-1}$ is a lens space.  This fact follows immediately from Theorem \ref{t: main}.

%%%
%%%
%%%

\subsection{Reducible surgeries.} When can surgery along a knot $K$ produce a reducible 3-manifold?  The cabling conjecture of Gonzalez-Acu\~na -- Short asserts that this can only occur when the knot is a cable knot, with the surgery slope provided by the cabling annulus \cite[Conjecture A]{gs:cabling}, \cite[Problem 1.79]{kirby:problems}.  From this it would follow that the surgery slope is an integer, and the reducible manifold has two prime summands, one of which is a lens space.

Analogous to the cyclic surgery theorem in this context, Gordon-Luecke proved that the surgery slope of a reducible surgery is an integer $p$, which we can again take to be positive upon reflecting the knot \cite{gl:reducible}.  They also proved that $K_p$ has a lens space summand \cite[Theorem 3]{gl:knotcomplements}.  In this vein, further work of Howie, Sayari, and Valdez S{\'a}nchez implies that $K_p$ has at most three prime summands, and if it has three, then two are lens spaces of coprime orders and the third is a homology sphere \cite{howie:free, sayari:bridge,valdezsanchez:fillings}.

Apparently unknown to practitioners of Floer homology, a bound strikingly opposite to \eqref{e: kmos} holds in this setting.  Building on work of Hoffman, Matignon-Sayari showed that if $K_p$ is reducible, then either $K$ and $p$ satisfy the conclusions of the cabling conjecture, or else \begin{equation}\label{e: ms}  p \leq 2g(K)-1\end{equation} \cite{hoffman,matignon-sayari}. Note that if $K$ is a cable knot with cabling slope $p$, then there is no relation in general between $p$ and $g(K)$.  On the other hand, assuming the surgered manifold is an L-space, we have the following easy result.

\begin{prop}\label{p: cable}

If $K$ is a cable knot with cabling slope $p$, and $K_p$ is an L-space, then $2g(K)-1 < p$.

\end{prop}

\begin{proof}

Let $K = C_{q,r}(\kappa)$ denote the cable knot, where $|q| \geq 2$.  Thus, $p = qr$ and $K_p \cong \kappa_{r/q} \# (- L(q,r))$.  In order for $K_p$ to be an L-space, $\kappa_{r/q}$ must be as well, so \eqref{e: kmos} implies that \[ 2g(\kappa) - 1 < r/q;\] the inequality is strict since the left side is an integer while the right side is not. On the other hand, an elementary calculation shows that \[ 2g(K)-1 = qr + q(2g(\kappa)-1)-r.\]  Thus, $2g(K) - 1 < qr = p$, as desired.

\end{proof}

Thus, in light of the Matignon-Sayari bound \eqref{e: ms} and Proposition \ref{p: cable}, in order to establish the cabling conjecture under the assumption that the surgered manifold is an L-space, it suffices to show that $K_{2g(K)-1}$ is never a reducible L-space.  Theorem \ref{t: main} shows that this is the case if we further assume that the surgered manifold bounds a negative definite 4-manifold with torsion-free $H_1$.

\begin{cor}\label{c: cor}

Suppose that $K_p$ is a reducible L-space and it bounds a smooth, negative definite 4-manifold $X$ with $H_1(X;\Z)$ torsion-free.  Then $K$ is a cable knot with cabling slope $p$. \qed

\end{cor}

\noindent In particular, Corollary \ref{c: cor} applies to a connected sum of lens spaces, a natural case of interest in view of the fact that any reducible surgery has a lens space summand. Accordingly, the cabling conjecture follows in this case.  A quick appeal to \cite[$\S$3]{gl:reducible} fills in the details of the following result.

\begin{thm}\label{t: cabling}

Suppose that surgery along a knot $K \subset S^3$ produces a connected sum of lens spaces.  Then $K$ is either a $(p,q)$-torus knot or a $(p,q)$-cable of an $(r,s)$-torus knot with $p = qrs \pm 1$, and the surgery slope is $pq$. The surgered manifold is $L(p,q) \# L(q,p)$ or $L(p, qs^2) \# L(q, \pm 1)$, respectively, both taken with the {\em opposite} orientation. \qed

\end{thm}

\subsection{Conventions.} We use homology groups with integer coefficients throughout.  All 4-manifolds are assumed smooth.  For a compact 4-manifold $X$, regard $H_2(X)$ as an inner product space equipped with the intersection pairing on $X$.  Let $W_{\pm p}(K)$ denote the 4-manifold obtained by attaching a $\pm p$-framed 2-handle to $D^4$ along the knot $K \subset S^3 = \del D^4$.

%%%%%%%%%%
%%%%%%%%%%
%%%%%%%%%%

\section*{Acknowledgments}

Thanks to John Baldwin for his advice on this paper.  Thanks to Cameron Gordon and John Luecke for enjoyable discussions and for drawing my attention to the references \cite{hoffman, matignon-sayari}.

%%%%%%%%%%
%%%%%%%%%%
%%%%%%%%%%

\section{Input from Floer homology}\label{s: input}

In \cite{os:absgr}, Ozsv\'ath-Szab\'o associated a numerical invariant $d(Y,\t) \in \Q$ called a {\em correction term} to an oriented rational homology sphere $Y$ equipped with a $\spc$ structure $\t$.  It is analogous to Fr{\o}yshov's $h$-invariant in monopole Floer homology \cite{froyshov:h}.  They proved that this invariant obeys the relation $d(-Y,\t) = - d(Y,\t)$, and that if $Y$ is the boundary of a negative definite 4-manifold $X$, then \begin{equation}\label{e: d bound} c_1(\s)^2 + b_2(X) \leq 4d(Y,\t) \end{equation} for every $\s \in \Spc(X)$ which extends $\t \in \Spc(Y)$ \cite[Theorem 9.6]{os:absgr}.

\begin{defin}\label{d: sharp}

A negative definite 4-manifold $X$ is {\em sharp} if, for every $\t \in \Spc(Y)$, there exists some extension $\s \in \Spc(X)$ that attains equality in the bound (\ref{e: d bound}).

\end{defin}

The following result provides the examples of L-spaces and sharp 4-manifolds that we will need.

\begin{prop}[Proposition 3.3 and Theorem 3.4, \cite{os:doublecover}]\label{p: sigma(K)}

Let $L$ denote a non-split alternating link.  Then the branched double-cover $\Sigma(L)$ is an L-space\footnote{No relation between the ``L"'s!}, and there exists a sharp 4-manifold $X$ with $\del X = \Sigma(L)$ and $H_1(X) = 0$.

\end{prop}

Every lens space $L(p,q)$ arises as the branched double-cover of a 2-bridge link.  In this case, the 4-manifold $X(p,q)$ implied by Proposition \ref{p: sigma(K)} admits the following description.  Assume $p > q > 0$, and write $p/q = [a_1,\dots,a_n]^-$ as a Hirzebruch-Jung continued fraction, with each $a_i \geq 2$. Then $X(p,q)$ denotes plumbing along a linear chain of disk bundles over $S^2$ with Euler numbers $-a_1,\dots,-a_n$, in that order.  From this perspective, the sharpness of $X(p,q)$ also follows from \cite[Corollary 1.5]{os:plumbed}.  In particular, $W_{-p}(U)$ is sharp, since it is diffeomorphic to the disk bundle of Euler number $-p$ over $S^2$.

In order to make use of \eqref{e: d bound}, we must understand $\spc$ structures on $K_p$.  Given $\t \in \Spc(K_p)$, it extends to some $\s  \in \Spc(W_p(K))$ since $H_1(W_p(K)) = 0$.  The group $H_2(W_p(K))$ is generated by the class of a surface $\Sigma$ obtained by smoothly gluing the core of the handle attachment to a copy of a Seifert surface for $K$ with its interior pushed into $int(D^4)$.  The quantity $\langle c_1(\s), [\Sigma] \rangle + p$ is an even value $2i$ whose residue class $(\mod 2p)$ does not depend on the choice of extension $\s$.  The assignment $\t \mapsto i$ sets up a 1-1 correspondence \begin{equation}\label{e: spc} \Spc(K_p) \overset{\sim}{\longrightarrow} \Z / p \Z.\end{equation}

Next, suppose that $K_p$ bounds a smooth, negative definite 4-manifold $X$ with $n := b_2(X)$.  The manifold $W := - W_p(K) \cong W_{-p}(\overline{K})$ is negative definite and has boundary $\overline{K}_{-p} = - K_p$, where $\overline{K}$ denotes the mirror image of $K$.  Form the closed, smooth, oriented 4-manifold $Z:= X \cup W$.  Since $b_1(K_p) = 0$, it follows that $b_2(Z) = b_2(X) + b_2(W) = n+1$; and since $H_2(X) \oplus H_2(W) \hookrightarrow H_2(Z)$, it follows that $Z$ is negative definite.  In particular, the square of a class in $H^2(Z)$ equals the sum of the squares of its restrictions to $H^2(X)$ and $H^2(W)$.

\begin{lem}\label{l: bound}

Suppose that $K_p$ bounds a smooth, negative definite 4-manifold $X$ with $H_1(X)$ torsion-free, and form $Z = X \cup W$ as above.  Then every $i \in \Spc(K_p)$ extends to some $\s \in \Spc(Z)$, and \begin{equation}\label{e: bound} c_1(\s)^2 + (n+1) \leq 4d(K_p,i) - 4d(U_p,i). \end{equation} Furthermore, if $X$ is sharp, then for every $i$ there exists some extension $\s$ that attains equality in \eqref{e: bound}.

\end{lem}

\begin{proof}

The fact that every $\spc$ structure on $K_p$ extends across $Z$ follows from the fact that $H_1(X)$ and $H_1(W)$ are torsion-free.  Now fix some $i \in \Spc(K_p)$ and an extension $\s \in \Spc(Z)$.  From \eqref{e: d bound} we obtain \[c_1(\s|X)^2 + b_2(X) \leq 4 d(K_p,i).\]   Observe that the maximum value of $c_1(\s|W)^2 + 1$ does not depend on the knot $K$.  Since $W_{-p}(U)$ is sharp, it follows that this value equals $4 d(U_{-p},i)$.  Therefore, \[ c_1(\s|W)^2 + 1 \leq - 4d(U_p,1).\]  Summing these two inequalities results in \eqref{e: bound}.

We obtain equality in \eqref{e: bound} under the assumption that $X$ is sharp by taking an extension of $i$ to some $\s_X \in \Spc(X)$ that attains equality in \eqref{e: d bound} and gluing it to an extension $\s_W \in \Spc(W)$ that attains the value $-4d(U_p,i)$.

\end{proof}

Let $K$ denote an L-space knot.  We aim to use \eqref{e: bound} to obtain information about the knot genus.  Consider the Alexander polynomial of $K$, \[ \Delta_K(T) = \sum_{j=-g}^g a_j \cdot T^j, \quad g := \deg(\Delta_K),\] and define the {\em torsion coefficient} \[t_i(K) = \sum_{j \geq 1} j \cdot a_{|i|+j}.\] Since $K$ is an L-space knot, \cite[Theorem 1.2]{os:lens} implies that the knot Floer homology group $\widehat{HFK}(K)$ is uniquely determined by the Alexander polynomial $\Delta_K$.  In particular, the maximum Alexander grading in which this group is supported is equal to the degree $g$ of $\Delta_K$.  On the other hand, \cite[Theorem 1.2]{os:genus} implies that this grading equals the knot genus: \[ g = g(K). \]  Furthermore, \cite[Theorem 1.2]{os:lens} implies that the non-zero coefficients of the Alexander polynomial take values $\pm 1$ and alternate in sign, beginning with $a_g = 1$.  It follows that for all $i \geq 0$, the quantity \[ t_i(K) - t_{i+1}(K) = \sum_{j \geq 1} a_{i+j} \] is always $0$ or $1$, so the $t_i(K)$ form a sequence of monotonically decreasing, non-negative integers for $i \geq 0$.  Therefore, we obtain \begin{equation}\label{e: max torsion} t_i(K) = 0 \text{ if and only if } |i| \ge g(K).\end{equation}

Owens-Strle state the following result explicitly \cite[Theorem 6.1]{os:lattices}; it slightly extends the case $q=1$ of \cite[Theorem 1.2]{os:rational} (see also \cite[Theorem 8.5]{kmos} or the identical \cite[Corollary 7.5]{os:absgr}).

\begin{thm}\label{t: torsion}

Let $K$ denote an L-space knot and $p$ a positive integer.  Then the torsion coefficients and correction terms satisfy \begin{equation}\label{e: torsion and d} - 2t_i(K) = d(K_p,i) - d(U_p,i), \text{ for all } |i| \leq p/2. \end{equation}

\end{thm}

In \cite[p. 538]{kmos}, the stated version of Theorem \ref{t: torsion} is used in conjunction with \eqref{e: max torsion} to enumerate the lens spaces obtained by surgery along a knot $K$ with genus $g(K) \leq 5$.  By using this approach in tandem with Lemma \ref{l: bound}, we will obtain the estimates presented in Theorem \ref{t: main}.  To that end, we focus our attention to the left-hand side of \eqref{e: bound}.  Donaldson's theorem implies that $H_2(Z) \cong - \Z^{n+1}$, where $-\Z^{n+1}$ denotes the integer lattice equipped with {\em minus} the standard Euclidean inner product \cite{d:thma}.  Choose an orthonormal basis $\{e_0,\dots,e_n \}$ for $-\Z^{n+1}$: $\langle e_i, e_j \rangle = - \delta_{ij}$ for all $i,j$.  The first Chern class map \[c_1: \Spc(Z) \to H^2(Z)\] has image the set of {\em characteristic covectors} for the inner product space $H_2(Z)$.  Identify $H_2(Z) \cong H^2(Z)$ by Poincar\'e duality; then Donaldson's theorem implies that this set corresponds to \[ \Char(-\Z^{n+1}) = \left\{ \c = \sum_{i=0}^n \c_i e_i \; | \; \c_i \text{ odd for all } i \right\}. \]  Write \[ \sigma = \sum_{i=0}^n \sigma_i e_i \] for the image of the class $[\Sigma]$ under the inclusion $H_2(W) \hookrightarrow H_2(Z) \cong - \Z^{n+1}$.

With the preceding notation in place, the following Lemma follows on combination of Lemma \ref{l: bound} with Theorem \ref{t: torsion}.

\begin{lem}\label{l: c}

Let $K$ denote an L-space knot, and suppose that $K_p$ bounds a smooth, negative definite 4-manifold $X$ with $H_1(X)$ torsion-free.   Then \begin{equation}\label{e: c} \c^2 + (n+1) \leq -8 t_i(K) \end{equation} for all $|i| \leq p/2$ and $\c \in \Char(-\Z^{n+1})$ such that $\langle \c, \sigma \rangle + p \equiv 2i \; (\mod 2p)$.  Furthermore, if $X$ is sharp, then for every $|i| \leq p/2$ there exists $\c$ that attains equality in \eqref{e: c}. \qed

\end{lem}

%%%%%%%%%%%
%%%%%%%%%%%
%%%%%%%%%%%

\section{The genus bounds}\label{s: inequalities}

We proceed to establish the bounds appearing in Theorem \ref{t: main}.  Both bounds stem from the following result, whose proof and application are elementary.

\begin{prop}\label{p: genus}

Under the hypotheses of Lemma \ref{l: c}, \begin{equation}\label{e: genus}  2g(K) \leq p - |\sigma|_1, \end{equation} with equality if $X$ is sharp.  Here $|\sigma|_1$ denotes the $L^1$ norm $\sum_{i=0}^n |\sigma_i|$.

\end{prop}

\begin{proof}

Select a value $0 \leq i \leq p/2$ for which the left-hand side of \eqref{e: c} vanishes.  Hence there exists $\c \in \{ \pm 1 \}^{n+1} \subset - \Z^{n+1}$ for which $\langle \c, \sigma \rangle + p \equiv 2i \; (\mod 2p)$.  On the other hand, we have $|\langle \c, \sigma \rangle| \leq | \langle \sigma,\sigma \rangle | = p$, so that in fact $\langle \c, \sigma \rangle + p = 2i$.  Now, the assumption on $i$ and the non-negativity of the torsion coefficients together imply that $t_i(K) = 0$, so that in fact $i \geq g(K)$ by \eqref{e: max torsion}.  It follows that for all $\c \in \{ \pm 1 \}^{n+1}$, we have \begin{equation}\label{e: genus'} 2g(K) \leq \langle \c, \sigma \rangle + p.\end{equation}  The minimum value of the right-hand side of \eqref{e: genus'} is attained by the {\em sign vector} $s(\sigma)$, defined by \[ s(\sigma)_j := \begin{cases} +1, & \text{if } \sigma_j \geq 0; \\  -1, & \text{otherwise.} \end{cases} \]  For it, \eqref{e: genus'} produces the desired bound \eqref{e: genus}.  The equality under the assumption that $X$ is sharp follows immediately.

\end{proof}

The bound \eqref{e: genusbound 1} in Theorem \ref{t: main} follows at once from Proposition \ref{p: genus} and the trivial inequality $p = |\langle \sigma,\sigma \rangle | \leq |\sigma|_1^2$.  Now suppose that $X$ is sharp.  Then $2g(K) = p - |\sigma|_1$ by Proposition \ref{p: genus}, and its proof extends to show that for all $p-|\sigma|_1 \leq 2i \leq p$, there exists $\c \in \{ \pm 1 \}^{n+1}$ with $\langle \c, \sigma \rangle + p = 2i$.  Replacing any such $\c$ by its negative, we obtain this fact for all $p-|\sigma|_1 \leq 2i \leq p+|\sigma|_1$.  In other words, for all $-|\sigma|_1 \leq j \leq |\sigma|_1$ with $j \equiv p \equiv |\sigma|_1 \; (\mod 2)$, there exists $\c \in \{ \pm 1 \}^{n+1}$ for which $\langle \c, \sigma \rangle = j$.  By a change of basis of $-\Z^{n+1}$, we may assume that the vector $\sigma$ has the property that \[ 0 \leq \sigma_0 \leq \cdots \leq \sigma_n. \]  Write a vector $\c \in \{ \pm 1 \}^{n+1}$ in the form $(-1,\dots,-1) + 2 \chi$, where $\chi \in \{0,1\}^{n+1}$.  Then we obtain that for every $0 \leq k \leq |\sigma|_1$, there exists $\chi \in \{0,1\}^{n+1}$ for which $|\langle \chi, \sigma \rangle| = - \langle \chi, \sigma \rangle = k$.    In other words, for every such $k$, there exists a subset $S \subset \{ 0, \dots, n \}$ for which $\sum_{i \in S} \sigma_i = k$.

\begin{lem}\label{l: change}

Consider a sequence of integers $0 \leq \sigma_0 \leq \cdots \leq \sigma_n.$  For every value $0 \leq k \leq \sigma_1 + \cdots + \sigma_n$, there exists a subset $S \subset \{0,\dots,n\}$ such that $\sum_{i \in S} \sigma_i = k$ if and only \begin{equation}\label{e: change} \sigma_i \leq \sigma_0 + \cdots + \sigma_{i-1} + 1 \text{  for all  } 1 \leq i \leq n. \end{equation}

\end{lem}

If we imagine the $\sigma_i$ as values of coins, then Lemma \ref{l: change} provides a necessary and sufficient condition under which one can make exact change from them in any amount up to their total value.  We call such a vector $\sigma = (\sigma_0,\dots,\sigma_n)$ a {\em changemaker} (cf. \cite{greene:3braids}); the concept was apparently first introduced under the term {\em complete sequence} in \cite{brown:changemaker, hks:changemaker}.  Before proceeding to the proof of Lemma \ref{l: change}, we enunciate what we have just established.

\begin{thm}\label{t: changemaker}

Let $K \subset S^3$ denote an L-space knot and suppose that $K_p$ bounds a sharp 4-manifold $X$.  Then $H_2(X) \oplus H_2(W)$ embeds as a full-rank sublattice of $- \Z^{n+1}$, where $n = b_2(X)$ and the generator of $H_2(W)$ maps to a changemaker $\sigma$ with $\langle \sigma,\sigma \rangle = -p$. \qed

\end{thm}

\begin{proof}[Proof of Lemma \ref{l: change}]

($\implies$)  We proceed by induction on $n$,  The statement is obvious when $n = 1$. For the induction step, select any value $1 \leq k \leq \sigma_1 + \cdots + \sigma_n$, and pick the largest value $j$ for which $k \geq \sigma_0 + \cdots + \sigma_{j-1}+1$.  By \eqref{e: change}, $k -\sigma_j \geq 0$, and by the choice of $j$, we have $k - \sigma_j \leq \sigma_0 + \cdots + \sigma_{j-1}$. By induction on $n$, there exists $S' \subset \{1,\dots,j-1 \}$ (possibly the empty set) for which $\sum_{i \in S'} \sigma_i = k - \sigma_j$; now $S = S' \cup \{ j \}$ provides the desired subset with $\sum_{i \in S} \sigma_i = k$.

\noindent ($\impliedby$)  We establish the contrapositive statement.  If the inequality \eqref{e: change} failed for some $i$, then let $k$ denote the value $\sigma_0 + \cdots + \sigma_{i-1} + 1$.  For any $S \subset \{ 1,\dots,n \}$, either $j < i$ for all $j \in S$, in which case $\sum_{j \in S} \sigma_j < k$, or there exists some $j \in S$ with $j\geq i$, in which case $\sum_{j \in S} \sigma_j \geq \sigma_i > k$.  Therefore, there does not exist a subset $S$ such that $\sum_{j \in S} \sigma_j = k$.
\end{proof}

Returning to the case at hand, we appeal to Lemma \ref{l: change} and invoke the inequality \eqref{e: change} for each $1 \leq i \leq n$ to obtain the estimate \begin{eqnarray*}
(\sigma_1 + \cdots + \sigma_n + 1)^2
&=& 1 + \sum_{i=1}^n \sigma_i^2 + 2 \sigma_i (\sigma_0+\cdots+\sigma_{i-1}+1) \cr \\
&\geq& 1 + \sum_{i=1}^n 3 \sigma_i^2 = 3p + 1.
\end{eqnarray*}  It follows that $|\sigma|_1 \geq \sqrt{3p+1}-1$, and on combination with the equality in  \eqref{e: genus} we obtain the desired bound \eqref{e: genusbound 2}.

%%%%%%%%%%
%%%%%%%%%%
%%%%%%%%%%

\section{Iterated cables}\label{s: iterated cable}

In this section, we prove the final assertion of Theorem \ref{t: main}.

\begin{prop}\label{p: iterated cable}

Let $p_0 = 0$, and for $n \geq 0$, inductively define $a_{n+1} = 2p_n+1$ and $p_n = 2a_n-1$. Let $K_n$ denote the $(2,a_n)$-cable of the $(2,a_{n-1})$-cable of the $\quad \cdots \quad (2,a_1)$-cable of the unknot.  Then for all $n \geq 1$, $p_n$-surgery along $K_n$ is an L-space which bounds a sharp 4-manifold, and the bound in Equation \eqref{e: genusbound 2} is attained in this case.

\end{prop}

We proceed by constructing a sharp 4-manifold $X_n$ for which $\del X_n = \Sigma(\kappa_n)$ for a particular alternating knot $\kappa_n$.  It follows quickly from the presentation of $\kappa_n$ that $\Sigma(\kappa_n) = (K_n)_{p_n}$ for {\em some} knot $K_n$ and slope $p_n \in \Q$.  The bulk of the argument consists in identifying the pair $(K_n,p_n)$ with the one stated in the Proposition.

\begin{proof}

\begin{figure}
\centering
\includegraphics[width=4in]{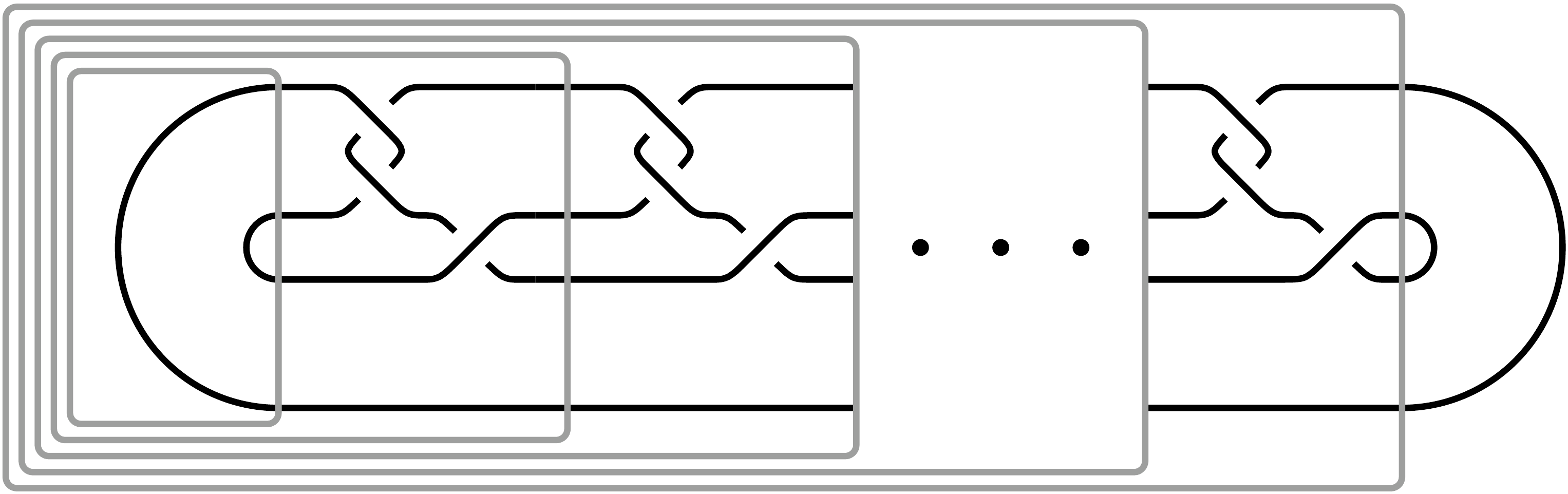}
\caption{The knot $\kappa_n$.}  \label{f: kappa}
\end{figure}

Let $\kappa_n$ denote the alternating knot depicted in Figure \ref{f: kappa}.  It contains $n$ copies of the tangle $\T$ displayed in Figure \ref{f: tangle}.  According to Proposition \ref{p: sigma(K)}, $\Sigma(\kappa_n)$ is an L-space for all $n \geq 1$, and there exists a sharp 4-manifold $X_n$ with $\del X_n = \Sigma(\kappa_n)$.  This space admits an explicit Kirby calculus description by attaching 2-handles along a framed link $\L_n \subset S^3 = \del D^4$.  Here $\L_n$ denotes a linear chain of $n-1$ unknots, with each component framed by $-5$ and oriented clockwise, and with each consecutive pair in the chain linked twice positively.  This Kirby description begins from the one described on \cite[p. 719]{os:unknotting}, replacing each 1-handle by a 0-framed unknot, and blowing down the $(-1)$-curves.

\begin{figure}
\centering
\includegraphics[width=3.5in]{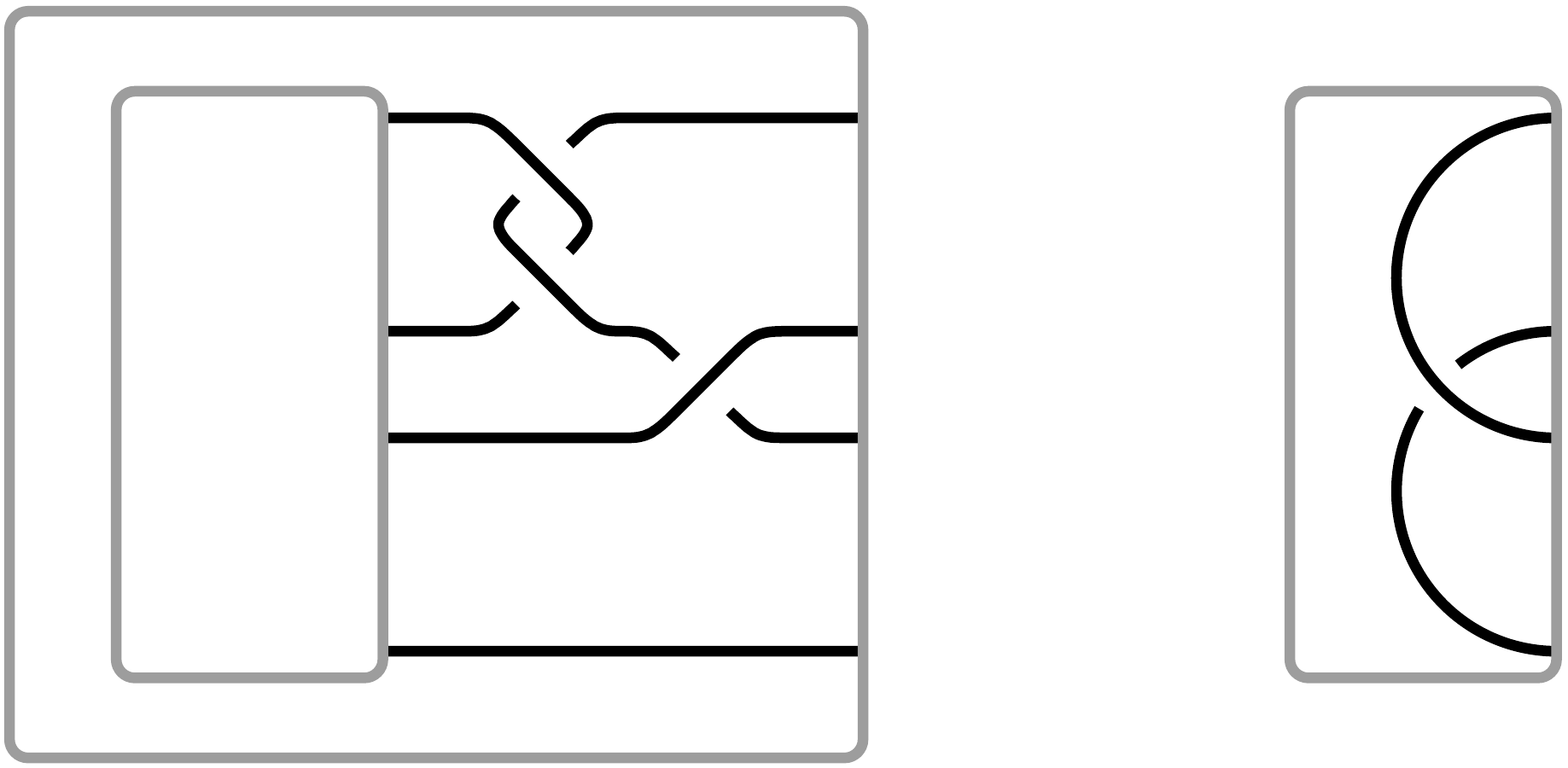}
\caption{A pair of tangles $\T$ and $\T'$.}  \label{f: tangle}
\end{figure}

The space $\Sigma(\T)$ is the (unique) Seifert-fibered space over the annulus with a single exceptional fiber of multiplicity 2.  Equivalently, it is homeomorphic to a $C_{2,q}$ cable space (here $q$ can denote any odd number).  In Figure \ref{f: tanglecurves} we redraw $\T$ with emphasis on a collection of arcs drawn on its boundary.  Filling along the preimage $\widetilde{\gamma_1}$ in $\Sigma(\T)$ produces a solid torus with meridian given by $\widetilde{\delta_1}$.  Observe that by filling $\T$ with the other tangle $\T'$ in Figure \ref{f: tangle}, we obtain a tangle isotopic as a marked tangle to $\T'$ itself.  Let $\T_n$ denote the complement to the inner-most tangle in the picture for $\kappa_n$.  By construction, one rational filling of $\T_n$ produces $\kappa_n$, while filling with $\T'$ produces the unknot.  It follows that the space $\Sigma(\T_n)$ is the complement of {\em some} knot $K_n \subset S^3$ for which $p_n$-surgery produces $\Sigma(\kappa_n)$ for {\em some} $p_n \in \Q$.  Identify the picture of $\T$ in Figure \ref{f: tanglecurves} with the inner-most copy appearing in the diagram for $\kappa_n$.

We claim that for all $n \geq 0$,

\begin{enumerate}

\item the pair $(K_n,p_n)$ agrees with the pair stated in the Proposition;

\item the curve $\widetilde{\gamma_1}$ represents a meridian $\mu$ for $K_n$; and 

\item the curve $\widetilde{\gamma_2}$ represents $p_n \cdot \mu + \lambda$, where $\lambda$ denotes the Seifert-framed longitude of $K_n$, and $\mu$ and $\lambda$ are oriented so that $\langle \mu, \lambda \rangle = + 1$.

\end{enumerate}

\begin{figure}
\centering
\includegraphics[width=2.5in]{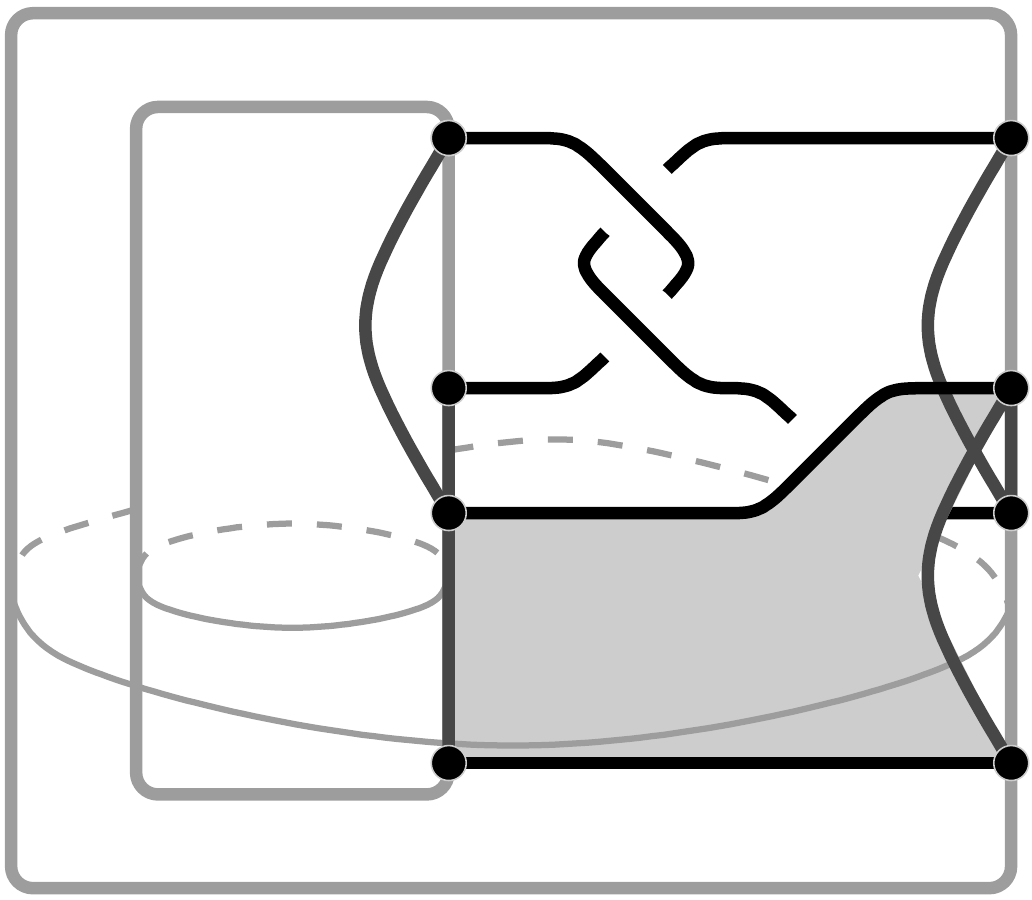}
\put(-130,100){$\gamma_1$}
\put(-90,75){$\gamma_2$}
\put(-115,40){$\gamma_3$}
\put(-30,100){$\delta_1$}
\put(0,75){$\delta_2$}
\put(-30,55){$\delta_3$}
\put(-60,45){$D$}
\caption{Arcs on the boundary of $\T$.}  \label{f: tanglecurves}
\end{figure}
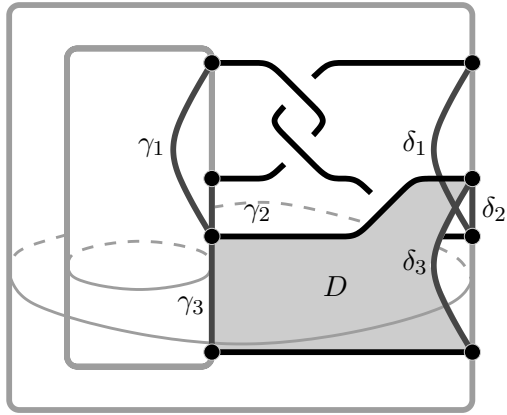

We proceed by induction on $n$.  When $n=0$, $\kappa_0$ is a two-component unlink, and $\Sigma(\kappa_0) = S^1 \times S^2$.  Assertions (1)-(3) follow easily by direct inspection.  Now assume that $n > 0$.  The space $\Sigma(\T_n)$ consists of filling $\Sigma(\T_{n-1})$ with the cable space $\Sigma(\T)$, where a meridian $\widetilde{\delta_1}$ of $\Sigma(\T)$ gets identified with the meridian $\mu'$ of $K_{n-1}$.  It follows at once that $\Sigma(\T_n)$ is the complement of some 2-cable of $K_{n-1}$; it stands to determine which precisely.  Observe that $\widetilde{\gamma_1}$ is a meridian $\mu$ for $K_n$ since filling along it produces $S^3$.  Also,  $\widetilde{\gamma_3}$ is a longitude for $K_n$ since it meets $\mu$ in a single point.  Furthermore, the annulus $\Sigma(D)$ connects $\widetilde{\gamma_3}$ with $\widetilde{\delta_3}$, which is a cable of $K_{n-1}$.  Let $\lambda'$ denote the Seifert-framed longitude of $K_{n-1}$, oriented so that $\langle \mu', \lambda' \rangle = +1$.  Then for one of the orientations on $\widetilde{\delta_3}$, we have \[ \langle  \mu', \widetilde{\delta_3} \rangle = \langle \widetilde{\delta_1}, \widetilde{\delta_3} \rangle = 2 \] and \[ \langle \widetilde{\delta_3}, \lambda' \rangle = \langle \widetilde{\delta_3},  \widetilde{\delta_2} - p_{n-1} \cdot \mu' \rangle = 1+ 2p_{n-1} = a_n. \]  Thus, $\widetilde{\delta_3}$ represents the class $a_n \cdot \mu' + 2 \lambda'$.  It follows that $K_n$ is isotopic to the $(2,a_n)$-cable of $K_{n-1}$.  To complete the induction step, we use the fact that $\widetilde{\gamma_3}$ represents the class $2a_n \cdot \mu + \lambda$ (cf. \cite[p. 32]{gordon:pcmi}).  Orienting $\widetilde{\gamma_2}$ appropriately, we have \[ \langle  \mu, \widetilde{\gamma_2} \rangle = \langle \widetilde{\gamma_1}, \widetilde{\gamma_2} \rangle = 1 \] and \[ \langle \widetilde{\gamma_2}, \lambda  \rangle = \langle \widetilde{\gamma_2},  \widetilde{\gamma_3} - 2 a_n \cdot \mu \rangle = -1 + 2a_n = p_n. \]  It follows that $\widetilde{\gamma_2}$ represents the stated class.

It stands to verify that $2g(K_n) -1 = p_n - \sqrt{3p_n + 1}$. This follows easily from the behavior of the knot genus under cabling.  An alternative argument runs as follows.  Since $b_2(X_n) = n-1$, the vector $\sigma$ belongs to $-\Z^n$.  Furthermore, $| \langle \sigma, \sigma \rangle | = p_n$.  In light of \eqref{e: change}, it follows that $\sigma = \sum_{i=1}^n 2^{i-1}e_i$.  The formula for $g(K_n)$  now follows on application of \eqref{e: genus}.

\end{proof}

Fintushel-Stern have given a construction for a Kirby diagram of an iterated cable \cite{fs:lenssurgery}.  It would be illuminating to identify the spaces $\Sigma(\kappa_n)$ and $(K_n)_{p_n}$ using their technique.

%%%%%%%%%%
%%%%%%%%%%
%%%%%%%%%%

\section{Concluding remarks}\label{s: conclusion}

\subsection{Iterated cables.}  We discovered the construction in Proposition \ref{p: iterated cable} in the following indirect way.  Suppose that $(K,p)$ attains equality in \eqref{e: genusbound 2}, where $K_p$ bounds a sharp 4-manifold $X$.  It follows that the vector $\sigma$ representing the class $[\Sigma]$ must attain equality in \eqref{e: change} for all $i$.  Thus, $\sigma$ takes the form $ \sum_{i=1}^n 2^{i-1} e_i$ for some $n \geq 1$, and $p = |\langle \sigma, \sigma \rangle | = p_n$.  Now, $H_2(X)$ embeds in $-\Z^n$ as the orthogonal complement $(\sigma)^\perp$.  This subspace is spanned by the vectors $2e_i - e_{i+1}$, for $i = 1,\dots,n-1$.  With respect to this basis, the intersection pairing on $X$ equals the linking matrix for $\L_n$.  Thus, the simplest choice for $X$ is the result of attaching 2-handles to $D^4$ along the framed link $\L_n$.  The knot $\kappa_n$ results from reverse-engineering the process for producing a sharp 4-manifold from the branched double-cover of a non-split alternating link \cite[p. 719]{os:unknotting}.  The family of knots $K_n$ follows in turn.

It appears difficult to address whether the family of knots $K_n$ attaining equality in \eqref{e: genusbound 2} is unique.  Any other candidate knot must have the same torsion coefficients, and hence knot Floer homology groups, as some $K_n$.  Examples of distinct L-space knots with identical knot Floer homology groups do exist, but not in great abundance (cf. \cite[$\S$1.1.3]{hedden:cabling2}).

%%%
%%%
%%%

\subsection{The Goda-Teragaito conjecture.}\label{ss: goda-teragaito}

Theorem \ref{t: berge bound} implies the second bound in \eqref{e: goda-teragaito} for all $p \geq 20$.  Furthermore, a quick analysis of changemakers of norm $18$ and $20$, coupled with an application of Proposition \ref{p: genus}, settles \eqref{e: goda-teragaito} for these two values of $p$ as well.  The values $p \leq 17$, with the exception of $p=14$, fall to a theorem of Baker \cite[Theorem 1.6]{baker:lens}.  Combining these results, the second bound in \eqref{e: goda-teragaito} follows for all except the two values $p \in \{14,19\}$.  Part of the difficulty in handling these remaining cases owes to the fact that $14$-surgery along the $(3,5)$-torus knot and $19$-surgery along the $(4,5)$-torus knot both produce lens spaces, while neither of these knots is hyperbolic.  The best that our methods establish is that any putative counterexample to \eqref{e: goda-teragaito} must have the same knot Floer homology groups as one of these two knots.

%%%
%%%
%%%

\subsection{The realization problem.}\label{ss: lens conclusion}  If $K_p = L(p,q)$ for some knot $K$, then Theorem \ref{t: changemaker} implies that $H_2(X(p,q)) \oplus H_2(-W_p(K))$ embeds as a full-rank sublattice of $-\Z^{n+1}$, and the vector $\sigma$ corresponding to the generator of $H_2(-W_p(K))$ is a changemaker $\sigma$.  Moreover, it follows in this case that $H_2(X(p,q))$ is the orthogonal complement to $(\sigma) \subset -\Z^{n+1}$.  This fact places a restriction on the intersection pairing of the plumbing manifold $X(p,q)$.  In fact, this necessary condition turns out to be sufficient as well.  This is the main thrust of \cite{greene:lens}, which answers the {\em realization problem}: which lens spaces arise by positive integer surgery along a knot $K \subset S^3$?  The refined techniques of that paper also lead to Theorem \ref{t: berge bound}.

%%%%%%%%%%%%%%%%
%%%%%%%%%%%%%%%%
%%%%%%%%%%%%%%%%

\bibliographystyle{plain}
\bibliography{/Users/Josh/Dropbox/Papers/References}

\end{document}